\documentclass[number,3p,preprint,9pt]{elsarticle}

\usepackage{amssymb}
\usepackage{amsmath}
\usepackage{amsthm}
\usepackage{mathtools}
\usepackage{tikz}
\theoremstyle{definition}
\newtheorem{theorem}{Theorem}

\newtheorem{remark}{Remark}
\newtheorem{observation}{Observation}
\newtheorem{proposition}{Proposition}
\newtheorem{claim}{Claim}

\usepackage{booktabs}

\usepackage{enumitem}
\setlist[itemize]{noitemsep}
\setlist[enumerate]{noitemsep}
\setlist[description]{noitemsep}

\journal{?}

\begin{document}

\begin{frontmatter}



\title{On the redundancy of transitivity constraints in the clique partitioning problem} 


\author{Noriyoshi Sukegawa} 

\affiliation{organization={Hosei Univercity},
            addressline={C2004, Koganei Campus, 3-7-2 Kajino-cho}, 
            city={Koganei-shi},
            postcode={184-8584}, 
            state={Tokyo},
            country={Japan}}

\begin{abstract}
In this study, we identify a class of redundant transitivity constraints in a 0-1 integer linear programming formulation of the clique partitioning problem.
The transitivity constraints in this class can be removed from the formulation without changing the optimal solution set, although each transitivity constraint defines a facet of the associated polytope.
This leads to a smaller formulation that is particularly effective for instances arising from correlation clustering, where edge weights are drawn from $\{-1,1\}$.
Our computational experiments show that the resulting formulation outperforms existing formulations on such instances.
\end{abstract}



\end{frontmatter}



\section{Introduction}\label{intro}

\subsection{Clique partitioning problem}
The clique partitioning problem (\texttt{CPP}) is a combinatorial optimization problem introduced by Gr\"otschel and Wakabayashi~\cite{grotschel1989cutting}. 
The \texttt{CPP} has applications in a wide range of industries and is recognized as a general framework for cluster analysis that captures models such as approximation by equivalence relations (Zahn's problem)~\cite{zahn1964approximating}, aggregation of equivalence relations (R\`egnier’s problem), modularity maximization~\cite{newman2004finding}, and correlation clustering~\cite{bansal2004correlation}. 

In the \texttt{CPP}, we are given a complete undirected graph $(V,E)$ with edge weights $w:E \to \mathbb{R}$. 
In applications to cluster analysis, vertices represent data points, and edge weights represent similarities between pairs of data points, where larger values indicate higher similarity. 
The objective is to partition the vertex set $V$ so that the total weight of the edges whose endpoints lie in the same part is maximized. 
The number of parts in the partition is not fixed, and hence the number of possible partitions (i.e., solutions of the \texttt{CPP}) is given by the $n$-th Bell number, where $n=|V|$.

The aforementioned special cases of the \texttt{CPP}, such as Zahn's problem and R\`egnier's problem, are all known to be NP-hard.
Therefore, the \texttt{CPP} is also NP-hard.
For the design of both exact and heuristic methods for the \texttt{CPP}, the following $0$--$1$ integer linear programming formulation, introduced in~\cite{grotschel1989cutting}, is commonly used; see, e.g.,~\cite{van2007correlation,agarwal2008modularity,nowozin2009solution,bruckner2013evaluation,jaehn2013new,miyauchi2013computing,sukegawa2013lagrangian}.
\begin{align*}
\begin{array}{llll}
\text{(P)}: &\text{max.}    &\displaystyle \sum_{\{i,j\} \in E} w_{ij} x_{ij}\\\\
            &\text{s. t.}  &x_{ij} + x_{jk} \le 1 + x_{ik} &\text{for all distinct $i,j,k \in V$}\\
            &                   &x_{ij}\in \{0,1\}                    &\text{for all distinct $i,j \in V$}. 
\end{array}
\end{align*}
Here, we use a $0$--$1$ variable $x_{ij}$ that takes $1$ if $i$ and $j$ are in the same part. 
The inequalities ensure that if $i$ and $j$, and $j$ and $k$, are in the same part, then $i$ and $k$ are also in the same part, and hence called transitivity constraints.
Since the number of transitivity constraints is $O(n^3)$, this formulation can be impractical for even medium-sized instances when solved by a general-purpose solver. 

\subsection{Motivation}
To address the size issue caused by the transitivity constraints, cutting-plane-based methods have been proposed and tested in the literature~\cite{grotschel1989cutting,ji2007branch,simanchev2019branch,letchford2024separation}. 
Accordingly, the facet structure of the associated polytope has also been  investigated~\cite{grotschel1990facets,oosten2001clique, letchford2025new}. 

Cutting-plane-based methods improve computational efficiency by excluding \emph{redundant} constraints that are automatically satisfied when maximizing the objective function.
In contrast, another line of research aims to identify redundant constraints based on the sign pattern of the objective function coefficients, initiated by Dinh and Thai~\cite{dinh2015toward} in the context of modularity maximization.
An extension to the \texttt{CPP} is given in~\cite{miyauchi2015redundant}. 
It states that removing all transitivity constraints of the form $x_{ij} + x_{jk} \le 1 + x_{ik}$ with $w_{ij} < 0$ and $w_{jk} < 0$ does not change the set of optimal solutions of (P). 
In other words, the set of optimal solutions of
\begin{align*}
\begin{array}{llll}
\text{(RP)}: &\text{max.}    &\displaystyle \sum_{\{i,j\} \in E} w_{ij} x_{ij}\\\\
            &\text{s. t.}  &x_{ij} + x_{jk} \le 1 + x_{ik} &\text{for all distinct $i,j,k \in V$ with $w_{ij}\ge 0$ or $w_{jk}\ge 0$ }\\
            &                   &x_{ij}\in \{0,1\}                    &\text{for all distinct $i,j \in V$}
\end{array}
\end{align*}
coincides with that of (P). 
Here, R denotes \emph{reduced}.
It was demonstrated in \cite{miyauchi2015redundant} that solving (RP) instead of (P) reduces solver time, especially for instances of modularity maximization, where many edges tend to have negative weights.
Later, it was shown in~\cite{miyauchi2018exact} that replacing $\ge$ with $>$ in the conditions of (RP) may yield an optimal solution that is infeasible for (P); nevertheless, such a solution can be transformed into an optimal solution for (P) via postprocessing that runs in $\mathrm{O}(n)$ time.
They also showed that this approach improves the performance of SAT-solver-based methods for the \texttt{CPP}~\cite{berg2013optimal}.

As mentioned in~\cite{miyauchi2015redundant,miyauchi2018exact}, an open question in this direction is whether the conditions can be strengthened to $w_{ij} + w_{jk} \ge 0$.
Note that $w_{ij} + w_{jk} \ge 0$ implies $w_{ij} \ge 0$ or $w_{jk} \ge 0$, and in particular, the gap between these conditions becomes significant in the presence of edges with large negative weights. 
According to the preliminary experimental results in~\cite{miyauchi2015redundant}, the answer to this open question appears to be affirmative; however, no formal proof has been provided to date.

\subsection{Our contribution}
In this paper, we prove the following, answers the aforementioned open question affirmatively. 
\begin{theorem}\label{thm1}
The set of optimal solutions of 
\begin{align*}
\begin{array}{llll}
\text{(FRP)}: &\text{max.}    &\displaystyle \sum_{\{i,j\} \in E} w_{ij} x_{ij}\\\\
            &\text{s. t.}  &x_{ij} + x_{jk} \le 1 + x_{ik} &\text{for all distinct $i,j,k \in V$ with $w_{ij}+w_{jk}\ge 0$ }\\
            &                   &x_{ij}\in \{0,1\}                    &\text{for all distinct $i,j \in V$}
\end{array}
\end{align*}
coincides with that of (P). 
\end{theorem}

Assume that the given edge weights $w$ are rational.  
Let $w'$ be the edge weights obtained by perturbing $w$, specifically by setting $w'_{ij} = w_{ij} - \epsilon$, where $\epsilon$ is a sufficiently small positive number. 
We observe that it is possible to choose $\epsilon$ so that 
\begin{itemize}[itemsep=0.5ex]
\item we never have $w'_{ij} + w'_{jk} = 0$ (thus, the condition $w'_{ij} + w'_{jk} \ge 0$ is equivalent to $w'_{ij} + w'_{jk} > 0$); 
\item any optimal solution with respect to $w'$ is also optimal with respect to $w$. 
\end{itemize}
This perturbation is expected to be particularly effective for instances arising in correlation clustering and group technology, where edge weights are drawn from $\{-1,1\}$. 

For example, when edge weights take these values uniformly at random, the condition $w_{ij} > 0$ or $w_{jk} > 0$ holds with probability $75\%$ for the modified version of (RP) introduced in~\cite{miyauchi2018exact}, which we denote by (mRP), whereas $w_{ij} + w_{jk} > 0$ holds with probability $25\%$ for (FRP) with the above perturbation preprocessing, which we denote by (pFRP). 

We conduct numerical experiments to investigate whether solving (pFRP) is faster than solving (mRP). 
The main findings are as follows.
\begin{enumerate}[itemsep=0.5ex]
\item Solving (pFRP) requires approximately half the time needed to solve (mRP) for graphs in which edge weights are drawn uniformly at random from $\{-1,1\}$. 
Furthermore, this advantage becomes more pronounced for graphs with structures that are well suited for clustering.
\item In contrast, for graphs typical in modularity maximization, where negative edge weights have relatively small magnitudes, the difference between (mRP) and (pFRP) is small. 
This is because, in such cases, if a constraint does not appear in (pFRP), i.e., $w_{ij} + w_{jk} \le 0$ holds, then it is likely that the constraint also does not appear in (mRP), i.e., both $w_{ij} \le 0$ and $w_{jk} \le 0$ hold. 
We also find that (pFRP) can be inferior to (mRP) in terms of solver time. 
In such cases, a different formulation, recently introduced in~\cite{koshimura2022concise} by Koshimura et al., is observed to outperform both (mRP) and (pFRP). 
\end{enumerate}

The remainder of this paper is organized as follows.
In Section~\ref{pre}, we examine properties of the optimal solutions to (FRP) that are used in the proof of Theorem~\ref{thm1}.
Based on these properties, we prove Theorem~\ref{thm1} in Section~\ref{proof}.
Section~\ref{ne} reports the results of numerical experiments.
Finally, Section~\ref{conc} concludes the paper.

\section{Preliminaries}\label{pre}

Let $\boldsymbol{x}$ be an arbitrary optimal solution to \text{(FRP)}.
Since \text{(FRP)} is a relaxation of \text{(P)}, it suffices to show that $\boldsymbol{x}$ is feasible for \text{(P)}. 
To this end, we look at the edge subset corresponding to $\boldsymbol{x}$; that is, 
\begin{align*}
X = \{ \{i,j\} \mid x_{ij}=1, \, i\neq j, \, i,j \in V \}. 
\end{align*}
We can conclude that $\boldsymbol{x}$ is feasible for (P) if $\{i,j\}, \{j,k\} \in X$ implies $\{i,k\} \in X$ for any distinct $i,j,k\in V$. 
Such an edge subset is referred to as the \emph{clique partitioning}  in the literature.  

To show that $X$ is a clique partitioning, it suffices to show that any connected component of $(V,X)$ is a complete subgraph (or clique) of $(V,X)$. 
To this end, take any connected component of $(V,X)$, and denote it by $C$. 
We write $(C,F)$ for the subgraph of $(V,X)$ induced by $C$, and for simplicity, with a slight abuse of notation, we write $w$ for the restriction of $w$ to $F$.

Let $(S,T)$ be any pair of subsets of $C$ with no intersection. 
We write $\delta(S,T)$ for the set of edges in $F$ with one endpoint in $S$ and the other in $T$; that is, 
\begin{align*}
\delta(S,T) = \{ \{i,j\} \in F \mid i \in S,\ j \in T \}
\end{align*}
where $F$ is omitted as it is always clear from the context.
Its weight is defined by 
\begin{align*}
w(S,T) = \sum_{\{i,j\} \in \delta(S,T)} w_{ij}. 
\end{align*}
We observe the following.  

\begin{observation}\label{obs1}
The graph $(C,F)$ with $w$, constructed as above, satisfies the following properties:
\begin{itemize}
\item[(C1)] The graph $(C,F)$ is connected.
\item[(C2)] The weight of any cut in $(C,F)$ is non-negative.
\item[(C3)] If $\{i,j\}, \{j,k\} \in F$ but $\{i,k\} \notin F$, then $w_{ij} + w_{jk} < 0$.
\end{itemize}
\begin{proof}
From the construction of $(C,F)$, properties (C1) and (C3) hold.
Suppose that there exists a cut $(S,T)$ with negative weight.
Let $\boldsymbol{x}'$ be the solution obtained from $\boldsymbol{x}$ by setting to zero all variables corresponding to edges in $\delta(S,T)$.
Then $\boldsymbol{x}'$ is feasible for \text{(FRP)}.
Indeed, any transitivity constraint of (FRP) violated by $\boldsymbol{x}'$ must involve vertices from both $S$ and $T$.
However, by the construction of $\boldsymbol{x}'$, among the three edges that appear in such a constraint, at most one can take $1$ in $\boldsymbol{x}'$.
Furthermore, since $w(S,T)<0$, the objective value of $\boldsymbol{x}'$ is strictly greater than that of $\boldsymbol{x}$, contradicting the optimality of $\boldsymbol{x}$ for (FRP).
\end{proof}
\end{observation}

\section{Proof of Theorem~\ref{thm1}}\label{proof}
\subsection{Strategy}

From Observation~\ref{obs1}, it suffices to prove the following Proposition~\ref{prop1} for all positive integers $k$ in order to establish Theorem~\ref{thm1}.

\begin{proposition}\label{prop1}
For any undirected graph $(C,F)$ with $|C|=k$ and edge weights $w$, if (C1), (C2), and (C3) hold, then $(C,F)$ is complete.
\end{proposition}

Our proof proceeds by induction on $k$.
The base case $k=1$ is trivial.
For the inductive step, let $(C, F)$ be an undirected graph with $|C| = n \ge 2$ and edge weights $w$ satisfying (C1), (C2), and (C3). 
Now, the goal is to prove that $(C,F)$ is complete. 
Assume, as the induction hypothesis, that Proposition~\ref{prop1} holds for all positive integers $k < n$.

\subsection{Outline}

Let $(S,T)$ be a minimum-weight cut of $(C,F)$, and assume that among all such cuts, it minimizes the number of edges crossing the cut; i.e., $|\delta(S,T)|$. 

Let $(S,F_S)$ denote the subgraph of $(C,F)$ induced by $S$. 
With a slight abuse of notation, we write $w$ for the restriction of $w$ to $F_S$.
Similarly, let $(T,F_T)$ denote the subgraph induced by $T$, and let $w$ denote the restriction of $w$ to $F_T$.

\begin{claim}\label{cl1}
Both the graphs $(S,F_S)$ and $(T,F_T)$ with edge weights $w$ satisfy (C1), (C2), and (C3).  
\begin{proof}
A proof is given in Section~\ref{proof_cl1}.
\end{proof}
\end{claim}

Since both $S$ and $T$ are nonempty, both $|S|$ and $|T|$ are strictly less than $n$. 
Hence, combining Claim~\ref{cl1} with the induction hypothesis, we see that $(S,F_S)$ and $(T,F_T)$ are both complete subgraphs of $(C,F)$. 

\begin{claim}\label{cl2}
Let $s \in S$ be a vertex that is adjacent to some but not all vertices in $T$; that is,
\begin{align*}
0 < |\delta(\{s\},T)| < |T|.
\end{align*}
Then, $w(\{s\}, T) < 0$.
\begin{proof}
A proof is given in Section~\ref{proof_cl2}.
\end{proof}
\end{claim}

Since $(C,F)$ satisfies (C1), there exists $s \in S$ with $|\delta(\{s\},T)| >0$. 
Among such vertices, there must exist at least one vertex $p \in S$ such that $|\delta(\{p\},T)| = |T|$. Otherwise, from Claim~\ref{cl2}, we obtain
\begin{align*}
w(S,T) = \sum_{s \in S} w(\{s\},T) < 0,
\end{align*}
which contradicts the assumption that (C2) holds for $(C,F)$. 
Fix such a vertex $p$. 
Note that $p$ is adjacent to every other vertex in $C$ 
because it is adjacent to every other vertex in $S$ due to the completeness of the subgraph $(S,F_S)$, and to every vertex in $T$ from $|\delta(\{p\},T)| = |T|$. 
By the symmetry, in a similar manner, we can show that there exists a vertex $q \in T$ that is adjacent to every other vertex in $C$. 

Consider the graph $(C',F')$ obtained from $(C,F)$ by contracting the two vertices $p$ and $q$ into a single vertex $pq$. 
Formally, $C' \coloneqq C \cup \{pq\} \setminus \{p,q\}$ and the edge set $F'$ is constructed as follows:
\begin{itemize}
\item for any distinct $i,j \in C'$ with $i \neq pq$ and $j \neq pq$, add $\{i,j\}$ to $F'$ if $\{i,j\} \in F$, and
\item  for any $j \in C' \setminus \{pq\}$, add $\{pq,j\}$ to $F'$.
\end{itemize}
In addition, define edge weights $w':F' \to \mathbb{R}$ by 
\begin{align*}
w_{ij}' = 
\begin{cases}
w_{ij} &\text{(if $i\neq pq$ and $j\neq pq$)}\\
w_{pj} + w_{qj} & \text{(if $i= pq$ and $j\neq pq$)}
\end{cases} ~~~\text{for each $\{i,j\} \in F'$.}
\end{align*}

\begin{claim}\label{cl3}
The graph $(C',F')$ with edge weights $w'$ satisfies (C1), (C2), and (C3).
\begin{proof}
A proof is given in Section~\ref{proof_cl3}. 
\end{proof}
\end{claim}

It follows from the construction of $(C',F')$ that $|C'| < n$.
Hence, combining Claim~\ref{cl3} with the induction hypothesis, we see that $(C',F')$ is complete. 
We now show that $(C,F)$ is complete. 
Suppose for contradiction that $\{i,j\} \notin F$ holds for some distinct $i,j \in C$. 
Since both $p$ and $q$ are adjacent to all other vertices in $C$, neither $i$ nor $j$ is equal to $p$ or $q$. 
In this case, $\{i,j\} \notin F$ implies $\{i,j\} \notin F'$ by the definition of $F'$.
However, this contradicts the completeness of $(C',F')$.
Therefore, $(C,F)$ is complete, which completes the proof of Theorem~\ref{thm1}.

\subsection{Proof of the claims}

\subsubsection{Claim~\ref{cl1}}\label{proof_cl1}
By symmetry, it suffices to prove the assertion for $(T,F_T)$. 
Since (C3) holds for $(C,F)$ with $w$, (C3) also holds for $(T,F_T)$ with $w$. 
Take any cut $(A,B)$ of $(T,F_T)$. 
Now, consider two cuts $(S \cup A, B)$ and $(S \cup B, A)$ of $(C,F)$. 
The weights of these two cuts satisfy
\begin{align*}
w(S,B) + w(A,B) &= w(S \cup A, B) \ge w(S,T)\\
w(S,A) + w(A,B) &= w(S \cup B, A) \ge w(S,T)
\end{align*}
from the minimality of the weight of the cut $(S,T)$ for $(C,F)$ with $w$. 
Noting that $w(S,T)=w(S,A)+w(S,B)$, these two inequalities imply 
\begin{align*}
2 \cdot w(A,B) \ge w(S,T) \ge 0
\end{align*}
where the last inequality follows from the assumption that (C2) holds for $(C,F)$ with $w$. 
Hence, (C2) holds for $(T,F_T)$ with $w$. 
Finally, we show that (C1) holds for $(T,F_T)$. 
Suppose for contradiction that there exists a cut $(A,B)$ such that no edges exist between $A$ and $B$, which also implies $w(A,B)=0$. 
In this case, there must exist edges between $S$ and $A$, and between $S$ and $B$. 
Otherwise, $(C,F)$ is not connected and hence does not satisfy (C1). 
On the other hand, we observe the following:  
\begin{itemize}
\item If $w(S,A)>0$, then the minimality of the weight of $(S, T)$ is violated because
\begin{align*}
w(S \cup A, B) &= w(S,B) + w(A,B)\\ &< w(S,B) + w(S,A)\\ &= w(S,T).
\end{align*}
\item If $w(S,A)<0$, then (C2) does not hold for $(C,F)$ with $w$ because the weight of the cut $(S \cup B, A)$ is given by $w(S,A)+w(A,B)=w(S,A)$. 
\end{itemize}
Therefore, we must have $w(S,A)=0$. 
This implies that $(S \cup A, B)$ has the same weight as $(S,T)$ but has fewer crossing edges, contradicting the choice of $(S,T)$.

\subsubsection{Claim~\ref{cl2}}\label{proof_cl2}
Let $s\in S$ be any vertex satisfying the condition of the claim. 
To simplify the notation, we write $T_s$ for the set of vertices $u\in T$ with $\{s,u\} \in F$. 
Since $0<|\delta(\{s\},T)|<|T|$, we have $T_s \neq \emptyset$ and $T \setminus T_s \neq \emptyset$. 

Now, focus on the cut $(T_s, T \setminus T_s)$ of $(T,F_T)$. 
For each pair of vertices $t \in T_s$ and $u \in T \setminus T_s$, from the definition of $T_s$ and the completeness of $(T,F_T)$, we have 
$\{s,t\}, \{t,u\} \in F  \mbox{ and } \{s,u\} \notin F$. 
This implies that for each pair of vertices $t \in T_s$ and $u \in T \setminus T_s$, we have
\begin{align}\label{myeq}
w_{st} + w_{tu} < 0
\end{align}
because (C3) holds for $(C,F)$. 
Then, for each vertex $u \in T\setminus T_s$, we observe 
\begin{align}\label{myeq2}
w(\{s\},T_s) + w(T_s,\{u\}) < 0
\end{align}
by summing up $(\ref{myeq})$ for all $t \in T_s$. 
Finally, summing up $(\ref{myeq2})$ for all $u \in T\setminus T_s$, 
\begin{align*}
|T\setminus T_s| \cdot w(\{s\},T_s) + \sum_{u \in T\setminus T_s} w(T_s,\{u\}) < 0
\end{align*}
follows. 
The second term on the left-hand side of this inequality coincides with $w(T_s, T \setminus T_s)$. 
Combining with the fact that (C2) holds for $(T,F_T)$, this term must be nonnegative. 
This implies that the first term on the left-hand side of this inequality is negative. 
It follows that $w(\{s\},T_s)<0$ as $T\setminus T_s \neq \emptyset$. 
This completes the proof of the claim since $\delta(\{s\},T_s) = \delta(\{s\},T)$ and hence $w(\{s\},T_s) = w(\{s\},T)$. 

\subsubsection{Claim~\ref{cl3}}\label{proof_cl3}
First, since $(C,F)$ is connected, $(C',F')$ is also connected. 
Hence, $(C',F')$ satisfies (C1).
Next, we verify that $(C',F')$ satisfies (C2).
By the construction of $(C',F')$, the weight of any cut in $(C',F')$ coincides with the weight of some cut $(S,T)$ in $(C,F)$ in which $p$ and $q$ belong to the same part of the cut; that is $p,q\in S$ or $p,q\in T$. 
It follows that $(C',F')$ satisfies (C2) because $(C,F)$ satisfies (C2). 
Finally, we verify that $(C',F')$ satisfies (C3).
To this end, take any distinct $i,j,k\in C'$, and suppose that 
\begin{align*}
\{i,j\}, \{j,k\} \in F' \text{ and } \{i,k\} \notin F'
\end{align*}
holds. 
Now, the goal is to show $w'_{ij}+ w'_{jk} <0$. 

\paragraph{Case 1}
Suppose that none of $i,j,k$ is the contracted vertex $pq$. 
In this case, $w'_{ij}=w_{ij}$ and $w'_{jk}=w_{jk}$ hold by the construction of $w'$. 
Since (C3) holds for $(C,F)$ with $w$, we have
\begin{align*}
w'_{ij}+ w'_{jk} = w_{ij}+ w_{jk} <0.
\end{align*}

\paragraph{Case 2}
Suppose that one of $i,j,k$ is the contracted vertex $pq$. 
It must be that $j = pq$ since $pq$ is adjacent to every other vertex of $C'$ in $(C',F')$. 
This implies $\{i,k\}\notin F$ from the construction of $F'$. 
On the other hand, 
$\{i,p\},\{p,k\},\{i,q\},\{q,k\}\in F$
since $p$ and $q$ are adjacent to every other vertex of $C$ in $(C,F)$. 
Combining with the fact that $(C,F)$ satisfies (C3) for $w$, we observe
\begin{align*}
w_{ip}+w_{pk} <0 \text{ and } w_{iq}+w_{qk} <0. 
\end{align*}
Since 
$w'_{ij}+ w'_{jk} = w'_{i,pq}+ w'_{pq,k} = w_{ip} + w_{iq} + w_{pk}+w_{qk}$, 
we have $w'_{ij}+ w'_{jk}<0$. 


\section{Numerical Experiments}\label{ne}

\begin{table}[t]
\centering
\small
\begin{minipage}{0.49\linewidth}
\begin{tabular}{lrrr}
\toprule
ID&\multicolumn{1}{c}{(mRP)}&\multicolumn{1}{c}{(pCP)}&\multicolumn{1}{c}{(Ours)}\\
\midrule
1 & 16.761 (9274) & 11.616 (6156) & \textbf{8.317} (\textbf{3158}) \\
2 & 11.879 (9050) & 10.864 (6043) & \textbf{6.321} (\textbf{2934}) \\
3 & 22.415 (9698) & 20.613 (6649) & \textbf{14.629} (\textbf{3574}) \\
4 & 24.276 (8367) & 10.041 (5115) & \textbf{6.302} (\textbf{2329}) \\
5 & 19.198 (9179) & 13.644 (6152) & \textbf{8.641} (\textbf{3029}) \\
6 & 36.634 (9522) & 39.155 (6517) & \textbf{23.777} (\textbf{3358}) \\
7 & 53.297 (9220) & 38.798 (6207) & \textbf{26.309} (\textbf{3156}) \\
8 & 17.413 (8843) & 12.301 (5638) & \textbf{8.612} (\textbf{2749}) \\
9 & 17.086 (9339) & 12.414 (6275) & \textbf{6.367} (\textbf{3261}) \\
10 & 17.991 (9269) & 11.754 (6265) & \textbf{7.161} (\textbf{3107}) \\
\bottomrule\\
\multicolumn{4}{l}{(a) Random}
\end{tabular}
\end{minipage}
\hfill
\begin{minipage}{0.49\linewidth}
\begin{tabular}{lrrr}
\toprule
ID&\multicolumn{1}{c}{(mRP)}&\multicolumn{1}{c}{(pCP)}&\multicolumn{1}{c}{(Ours)}\\
\midrule
1 & 14.022 (7196) & \textbf{8.501} (\textbf{4273}) & 8.638 (4492) \\
2 & 17.569 (6615) & 9.544 (\textbf{3876}) & \textbf{7.817} (4089) \\
3 & 5.665 (6680) & \textbf{3.736} (\textbf{3917}) & 4.165 (4266) \\
4 & 4.611 (6758) & \textbf{3.082} (\textbf{4073}) & 3.814 (4287) \\
5 & 1.374 (6433) & \textbf{0.719} (3779) & 1.045 (\textbf{3733}) \\
6 & 4.492 (7108) & 2.741 (4264) & \textbf{2.702} (\textbf{4211}) \\
7 & 7.355 (6954) & \textbf{4.868} (4099) & 5.261 (\textbf{3827}) \\
8 & 5.936 (6840) & 4.294 (\textbf{4232}) & \textbf{3.616} (4275) \\
9 & 12.356 (7136) & \textbf{7.824} (\textbf{4241}) & 8.233 (4499) \\
10 & 4.968 (6591) & \textbf{3.265} (\textbf{3643}) & 4.009 (3977) \\
\bottomrule\\
\multicolumn{4}{l}{(b) Sparse}
\end{tabular}
\end{minipage}

\vspace{1em}

\begin{minipage}{0.49\linewidth}
\begin{tabular}{lrrr}
\toprule
ID&\multicolumn{1}{c}{(mRP)}&\multicolumn{1}{c}{(pCP)}&\multicolumn{1}{c}{(Ours)}\\
\midrule
1 & 1.397 (6999) & 0.493 (4101) & \textbf{0.281} (\textbf{1513}) \\
2 & 1.929 (7355) & 0.993 (4418) & \textbf{0.420} (\textbf{1661}) \\
3 & 0.497 (6871) & 0.262 (3976) & \textbf{0.193} (\textbf{1361}) \\
4 & 2.396 (6683) & 1.568 (4018) & \textbf{1.087} (\textbf{1213}) \\
5 & 3.408 (6638) & 0.968 (4011) & \textbf{0.535} (\textbf{1314}) \\
6 & 3.944 (7032) & \textbf{0.801} (4106) & 0.845 (\textbf{1480}) \\
7 & 0.949 (6557) & 0.541 (3814) & \textbf{0.300} (\textbf{1227}) \\
8 & 5.373 (7442) & 1.342 (4744) & \textbf{0.763} (\textbf{1686}) \\
9 & 2.818 (7156) & 1.065 (4239) & \textbf{0.628} (\textbf{1524}) \\
10 & 2.413 (7113) & 0.958 (4192) & \textbf{0.508} (\textbf{1511}) \\
\bottomrule\\
\multicolumn{4}{l}{(c) Structured}
\end{tabular}
\end{minipage}
\hfill
\begin{minipage}{0.48\linewidth}
\begin{tabular}{lrrr}
\toprule
ID&\multicolumn{1}{c}{(mRP)}&\multicolumn{1}{c}{(pCP)}&\multicolumn{1}{c}{(Ours)}\\
\midrule
1 & 0.129 (2897) & \textbf{0.064} (\textbf{1157}) & 0.112 (2844) \\
2 & 0.143 (2721) & \textbf{0.086} (\textbf{1104}) & 0.214 (2617) \\
3 & 0.087 (2792) & \textbf{0.046} (\textbf{1101}) & 0.078 (2693) \\
4 & 0.272 (2772) & \textbf{0.176} (\textbf{1027}) & 0.232 (2643) \\
5 & 0.296 (2824) & 0.274 (\textbf{1059}) & \textbf{0.265} (2750) \\
6 & 0.329 (2891) & \textbf{0.152} (\textbf{1196}) & 0.250 (2764) \\
7 & 0.276 (2840) & \textbf{0.164} (\textbf{1095}) & 0.288 (2767) \\
8 & 0.132 (2918) & \textbf{0.070} (\textbf{1156}) & 0.148 (2901) \\
9 & 0.189 (2917) & \textbf{0.072} (\textbf{1132}) & 0.134 (2889) \\
10 & 0.214 (2742) & \textbf{0.137} (\textbf{1092}) & 0.217 (2604) \\
\bottomrule\\
\multicolumn{4}{l}{(d) Modularity}
\end{tabular}
\end{minipage}

\caption{Runtime (seconds) with the number of constraints in parentheses. The fastest runtime and smallest constraint count are highlighted independently.}
\label{tab:runtime_constraints}
\end{table}

\subsection{Compared Formulations}

In this section, we report the computational performance of (pFRP).
As explained in Section~\ref{intro}, (pFRP) is obtained from (FRP) by perturbing the edge weights to reduce the number of transitivity constraints.
In what follows, this formulation is referred to as (Ours).

As a benchmark, we consider the two formulations explained in Section~\ref{intro}, namely those proposed in \cite{miyauchi2018exact} and \cite{koshimura2022concise}. 
The former is denoted by (mRP) while the latter is denoted by (CP), where C stands for \emph{concise}. 
In (CP), for each triple of distinct vertices $i,j,k \in V$, 
\begin{align*}
x_{ij} + x_{jk} &\le 1 + x_{ik} \quad \text{if } w_{jk} \ge 0,\\
x_{ij} + x_{ik} &\le 1 + x_{jk} \quad \text{if } w_{ik} \ge 0,\\
x_{ik} + x_{jk} &\le 1 + x_{ij} \quad \text{if } w_{ij} \ge 0
\end{align*}
are included as its transitivity constraints.
As in (pFRP), the edge weights can be perturbed so that $\ge$ can be replaced with $>$ in each condition, thereby reducing the number of transitivity constraints.
We denote the resulting formulation by (pCP) and include it in our experiments.

\subsection{Instances}

We consider four types of instances. 

\begin{itemize}
\item Random: Edge weights are independently drawn from $\{-1,1\}$ with equal probability. 
\item Sparse: Edge weights are independently drawn from $\{-1,0,1\}$ with equal probability.
\item Structured: The vertex set is partitioned into five clusters of equal size. 
Edge weights are independently drawn from $\{-1,1\}$ with probability $25\%$ and $75\%$ if endpoints belong to the same cluster, and $75\%$ and $25\%$ if endpoints belong to different clusters.
\item Modularity: Edge weights are defined as an instance of modularity maximization generated from a random graph following the Barabási--Albert model.
\end{itemize}
As discussed in Section~1, the settings of Random and Sparse is typical in correlation clustering and group technology, where edge weights zero are interpreted as missing values in practice. 
These two instances are purely random and rarely arise in practice.
In contrast, structured instances are designed to reflect practical situations in which a certain degree of clusterable structure exists;
that is, in an optimal partition, positive weights are concentrated within clusters, whereas negative weights are concentrated between clusters.

\subsection{Results}

All experiments were conducted on a 64-bit Windows 11 PC with an 11th Gen Intel(R) Core(TM) i5-1145G7 CPU at 2.60 GHz and 16 GB of RAM.
All the formulations were implemented in Python, and the resulting 0--1 integer linear programming problems were solved using Gurobi version 12.1.1. 
Although no time limit is imposed, we exclude the standard formulation (P) from our experiments since solving it requires prohibitively large computation time.

The results are summarized in Table~\ref{tab:runtime_constraints}. 
We set $n=30$ and generated 10 instances for each type. 
For each formulation, we report the computation time by Gurobi in seconds, together with the number of transitivity constraints in parentheses.
For each instance ID, the shortest computation time and the smallest number of transitivity constraints are shown in bold.
The main observations from these results are as follows.

\begin{itemize}
    \item
    As mentioned in Section~\ref{intro}, for Random instances, the number of transitivity constraints in (Ours) is expected to be approximately one third of that in (mRP), and this is indeed observed.
    However, the computation time is not reduced to one third; rather, it is reduced to roughly one half.
    On the other hand, the difference from (pCP) is limited.

    \item
    For Sparse instances, (Ours) still outperforms (mRP).
    However, in many cases, (pCP) outperforms (Ours) in terms of both computation time and the number of transitivity constraints.
    This is because, if two edges with weights 0 and 1 appear on the left-hand side of a transitivity constraint, then this constraint is always included in (Ours), whereas in (pCP), depending on the ordering of the indices, it may be excluded.
    Note that as can also be seen from the table, there is no inclusion relationship between the set of transitivity constraints of (Ours) and (pCP).

    \item
    For Structured instances, (Ours) achieves the best performance in most cases.
    Compared with (mRP), the computation time is sometimes reduced to about one seventh.
    Compared with (pCP), it is reduced to about one half on average.

    \item
For Modularity instances, (pCP) achieves the best performance in most cases.
(Ours) is outperformed by (mRP) for several instances.
Due to the inclusion relationship, (Ours) never has more transitivity constraints than (mRP),
although the difference is small, as explained in Section~\ref{intro}.
From the viewpoint of computation time, however, (Ours) can be worse than (mRP). 
This is possibly because unnecessarily eliminating constraints may degrade the solver's performance.
\end{itemize}

\section{Concluding Remarks}\label{conc}

In this paper, we have presented a reformulation of the standard formulation for the clique partitioning problem.
Among the transitivity constraints that appear in the formulation,
we have characterized those that are automatically satisfied when maximizing the objective function and are therefore redundant.
This characterization depends on the sign pattern of the objective coefficients; i.e., the edge weights.

Our result shows that it suffices to consider only those satisfying $w_{ij}+w_{jk}\ge 0$. 
It remains open whether a stronger condition exists.
A natural candidate would be $w_{ij}\ge 0 \text{ and } w_{jk}\ge 0$; however, we readily see that it admits counterexamples. 
Previous studies have focused solely on the signs of edge weights appearing on the left-hand side of the constraints.
It may therefore be worthwhile to incorporate the edge weights on the right-hand side as well.
For instance, our preliminary experiments suggest that it may be sufficient to consider only those transitivity constraints that satisfy $w_{ij}+w_{jk}-w_{ik}\ge 0$, that is, those for which the inner product between the normal vector of the inequality and the objective vector is nonnegative, although, of course, this property does not hold for general integer programs. 

Existing formulations tend to reduce the number of transitivity constraints when many edge weights are negative.
Therefore, it would be interesting to investigate whether there exists a formulation that instead yields fewer transitivity constraints when many edge weights are positive. 
In addition, it would be worthwhile to explore characterizations based on global structural properties,
rather than relying solely on local information such as the signs of edge weights appearing in individual transitivity constraints.

\bibliographystyle{elsarticle-harv} 
\bibliography{cas-refs}

\end{document}